\documentclass[11pt]{amsart}
\usepackage{amsmath}
\usepackage{amsthm}
\usepackage{amsbsy}
\usepackage{amssymb}
\usepackage{amsfonts}
\usepackage[dvips]{lscape}
\usepackage{xcolor}
\usepackage{amscd}
\usepackage[all,cmtip]{xy}
\usepackage{euscript}
\usepackage{parskip}
\usepackage{enumerate}
\usepackage[all,cmtip]{xy}
\usepackage[foot]{amsaddr}
\usepackage{tikz-cd}

\usepackage[margin=1in]{geometry}
\usepackage[expansion=false]{microtype}

\usepackage[pdfusetitle,unicode,hidelinks]{hyperref}

\theoremstyle{plain}

\theoremstyle{plain}
\newtheorem{theorem}{Theorem}[section]

\newtheorem{lemma}[theorem]{Lemma}
\newtheorem{corollary}[theorem]{Corollary}

\makeatletter
\newtheorem*{rep@theorem}{\rep@title}
\newcommand{\newreptheorem}[2]{%
\newenvironment{rep#1}[1]{%
 \def\rep@title{#2 \ref{##1}}%
 \begin{rep@theorem}}%
 {\end{rep@theorem}}}
\makeatother

\theoremstyle{remark}
\newtheorem{remark}[equation]{Remark}

\theoremstyle{definition}
\newtheorem{definition}[theorem]{Definition}
\title{A bivariate approach to realrootedness of special polynomials}
\date{\today}
\author{Aurelien Gribinski\\ Jussieu, Inria (Paris)}
\email{axgribinski@gmail.com}
\begin{document}
\renewcommand{\contentsname}{}
\maketitle
\begin{abstract}

In this paper, we exhibit new monotonicity properties of roots of families of orthogonal polynomials $P_n^{(z)}(x)$ depending polynomially on a parameter (Laguerre and Gegenbauer).  Establishing that  $P_n^{(z)}(x)$ are realrooted in $z$ for $x$ in the support of orthogonality, we show realrootedness in $x$ and interlacing properties of $\partial_z^kP_n^{(z)}(x)$ for $k\leq n$ and $z$ in the appropriate interval, establishing a dual approach to orthogonality. 
\end{abstract}
\setcounter{tocdepth}{1}

\section{Introduction}\label{intro}
Orthogonal families of polynomials like generalized Laguerre and Gegenbauer polynomials have long been studied, and show up in all fields of mathematics and physics. But little has been said about the properties of such polynomials when we vary the underlying parameter (see Chapter~6 in  \cite{Szego} for a review). Asymptotic density of generalized Laguerre roots when parameters are negative was recently investigated in \cite{KUI}, pseudo orthogonality properties of Gegenbauer polynomials with negative parameters were considered in \cite{Alfa}, and asymptotic results on root densities with sliding parameter were obtained in \cite{KVA}. The most relevant attempt to change perspective, as we will see below, comes from Charlier polynomials. However, we believe that a bivariate systematic approach to such polynomials is still missing. \\
We study families of generalized orthogonal polynomials $P_n^{(z)}(x)$ depending on a parameter $z$ from a new angle, that is, as bivariate polynomials $P_n(x,z)$. Which allows us to fix for instance the usual variable $x$ and consider them instead as polynomials in $z$. In the following $P_n^{(z)}(x)$ will refer to generic orthogonal polynomials, either Gegenbauer or Laguerre (and possibly other families, to be determined).

Let us recall standard definitions.
\begin{definition}[Generalized Laguerre and Gegenbauer families]\label{defnfamily}
We define the generalized  Laguerre polynomials $L_n^{(z)}(x)$ (alternatively referred to as $L_n(x,z)$ in the paper) with parameter $z$ as the degree $n$ solution of Kummer's confluent hypergeometric differential equation :
\[
xy''(x)+(z+1-x)y'(x)+ny(x)=0.
\]
We have the polynomial representation: \[ L_n^{(z)}(x) = \sum_{k=0}^n \frac{(-1)^k  \prod_{j=k+1}^{n} (z+j)}{k!(n-k)!} x^k.  \]
Alternatively, we define the Gegenbauer (ultraspehrical) polynomials $G_n^{(z)}(x)$ (alternatively referred to as $G_n(x,z)$ in the paper) with parameter $z$ as the degree $n$ solution of the hypergeometric differential equation:
\[
(1-x^2)y''(x)-(2z+1)xy'(x)+n(n+2z)y(x)=0.
\]
They are explicitly given by: \[ G_n^{(z)}(x) = \sum_{k=0}^{\lfloor n/2 \rfloor} (-1)^k \frac{   \prod_{i=0}^{n-k-1}(z+i) }{k! (n-2k)!} (2x)^{n-2k}.    \]
In the following, we will  find it more appropriate to consider instead the modified family, that verifies the exact same differential equation:
\[
 \hat{G}_n^{(z)}(x):= P_n^{(z-1/2,z-1/2)}(x) = \frac{\prod_{i=0}^{n-1}(z+1/2+i)}{2^n\prod_{i=0}^{n-1}(z+i/2)} G_n^{(z)}(x).
 \]
\end{definition}
They correspond to Jacobi polynomials $ P_n$ with equal parameters, and are just a rescaling of Gegenbauer polynomials. As a consequence, the roots in $x$ are identical for $\hat{G}_n^{(z)}(x)$ and $G_n^{(z)}(x)$ (as functions of $z$). We have:
\[
 \hat{G}_n^{(z)}(x)= \sum_{k=0}^{\lfloor n/2 \rfloor} (-1)^k \frac{\prod_{i=n-\lfloor n/2 \rfloor}^{n-k-1}(z+i)\prod_{i=\lfloor n/2 \rfloor}^{n-1}(z+1/2+i)}{k! (n-2k)!} x^{n-2k}2^{-2k}.
 \]
These two families share many strong properties as polynomials in $x$ due to the common orthogonality property.

\begin{theorem}[Theorem~3.3.1 in \cite{Szego}] Take $z>-1$ in the Laguerre case and $z> -1/2$ in the Gegenbauer case.
 The $P_n^{(z)}(x)$ form families of orthogonal polynomials  with real distinct roots on respectively $[0,+\infty]$ (Laguerre) and $[-1,1]$(Gegenbauer), with respective orthogonality weight densities $e^{-x}x^z$ (Laguerre) and $(1-x^2)^{z-1/2}$ (Gegenbauer). 
\end{theorem}
\begin{remark}
Notice that there is an issue at $z=0$, as we have  $G_n^{(0)}(x)=0$, so that strictly speaking the Gegenbauer family isn't orthogonal ponctually, but by rescaling it using the above modified Gegenbauer family, we remove the singularity and make it orthogonal. 
\end{remark}

We can relate different degree polynomials of the family through the concept of interlacing; let us recall it's meaning:
\begin{definition}
Assume that a polynomial $p$ is of degree $n$ and has real roots $p_1\geq p_2\dots\geq p_n$, and $q$ of degree $n$ or $n-1$ has roots $q_1\geq q_2\dots\geq q_{n-1} (\geq q_n)$.We say that they are interlacing if 
\[
p_1 \geq q_1 \geq p_2 \geq q_2\dots.\geq p_{n-1}\geq q_{n-1}\geq p_n (\geq q_n)
\]
or if the equivalent inequality permuting the roles of $p$ and $q$ holds in case the polynomials have the same degree.
 We also say that the interlacing is strict when all inequalities above are strict, meaning that they have distinct simple roots. 
 \end{definition}

It is a classical result that:
\begin{theorem}[Interlacing in $x$, see Theorem~3.3.2 in \cite{Szego}]\label{interlderiv} For $z>-1$ in the Laguerre case and $z> -1/2$ in the Gegenbauer case (again for $z=0$ we consider the modified Gegenbauer family),
 $P_n^{(z)}(x)$ and $P_{n-1}^{(z)}(x)$ interlace strictly, as well as  $P_n^{(z)}(x)$ and $\partial_xP_{n}^{(z)}(x)$. 
\end{theorem}

We will extensively use the following theorem, describing the monotonous evolution of zeros when the parameter is moving.

\begin{theorem} [Monotonicity of the roots with respect to the parameter by Markov and Stieltjes, see Theorem~6.12.1 and Theorem~6.21.1 in \cite{Szego}] \label{monoroots} 
If $x_i(z)$ are the roots of $L_n^{(z)}(x)$, then for $z>-1$ they are increasing with $z$: $\frac{d}{dz}x_i(z) > 0$, $ i \leq n$ (in Szego's notations, $w(x,z)= e^{-x}x^z $ is such that $\frac{\partial_zw(x,z)}{w(x,z)}= log(x)$ which is increasing). Also, the positive roots $y_i(z)$ of  $G_n^{(z)}(x)$  are decreasing with $z$ for $z>-1/2 $:  $\frac{d}{dz}y_i(z) < 0$ (the negative  ones are increasing by symmetry). 
\end{theorem}

It follows from their hypergeometric expressions  that $L_n^{(z)}(x)$ and  $G_n^{(z)}(x)$ are of  degree exactly $n$ in $z$ (except for $x=0$ in the Gegenbauer case, see expressions in \ref{defnfamily}), allowing a duality approach. Switching the variable and the parameter, it was found in the Laguerre case that the polynomials thus obtained have similar properties to the original polynomials.

\begin{theorem} (Charlier polynomials, see p.44 and p.50 in \cite{DOP}) Define the following family of polynomials:
\[
C_n^{(x_0)}(z)= \frac{(-1)^n}{\mu^n}n! L_n^{(z-n)}(x_0).
\]
Then this family is discrete orthogonal for $x_0>0$, with the following relations, for all integers $m$ and $n$:
\[
\sum_{z=0}^{\infty} e^{-x_0}\frac{x_0^z}{z!}C_n^{(x_0)}(z)C_m^{(x_0)}(z)= \delta_{nm}n!x_0^{-n}.
\]

\end{theorem}
Using the fact (see Theorem~3.3.1 in \cite{Szego}) that polynomials from discrete orthogonal families have real distinct roots, we readily get the following.

\begin{corollary}\label{Laguerrez}
For fixed $x_0\in[0,+\infty[$, $L_n(x_0,z)$ is a real-rooted polynomial in $z$ of degree $n$ with simple roots.
\end{corollary}
\begin{remark}
The case $x_0=0$ doesn't follow from orthogonality but can be trivially checked by hand using the expression  in \ref{defnfamily}. 
\end{remark}

The same discrete orthogonality property can't be used to prove a similar realrootedness result for Gegenbauer polynomials. In fact, as we will see, there are constant roots in $z$ shared among polynomials of increasing degree ($G_n$ and $G_{n-1}$) and if there was a relation of orthogonality in $z$ the roots of successive polynomials would be distinct (Theorem~\ref{interlderiv} applies for discrete orthogonality too). Instead, we give a direct proof of realrootedness in Section~2.

\begin{theorem}\label{Gegenbauerz}
For fixed $x_0\in[-1,1]$, $G_n(x_0,z)$ is a real-rooted polynomial in $z$ of degree $n$ (except for $x_0=0$, but in this case it still has real roots, though fewer). Its nonconstant roots in $z$ ( $\lfloor n/2 \rfloor$ of them, see Definition~\ref{defngtilde}) are increasing for $x\in[-1,0[$, and decreasing for $x\in]0,1]$, with a convergence to $+\infty$ when $x \to 0$.
\end{theorem}

\begin{remark}
One can wonder if there exists an easy proof. We could think of using Theorem~\ref{monoroots}, investigate  branches  $(y_i(z),z)$  in the plane and invert the functions  by monotonicity to get roots $z_i(y)$. However there are roots in $z$ that are arbitrarily negative when the degree increases, and such monotonicity properties fail in this negative domain. 
\end{remark}

\begin{corollary}\label{Gegenbauerzmod}
For fixed $x_0\in[-1,1]$, $\hat{G}_n(x_0,z)$ is a real-rooted polynomial in $z$ of degree $n$ (except for $x_0=0$). Furthermore, it has distinct simple roots. Its nonconstant roots in $z$ ($\lfloor n/2 \rfloor$ of them) are increasing for $x\in[-1,0[$, and decreasing for $x\in]0,1]$, with a convergence to $+\infty$ when $x \to 0$.
\end{corollary}
These results are obtained (in Section~2) through new interlacing properties in both variables, which are intriguing dual versions of the classical Theorems~\ref{interlderiv}.

\begin{theorem}[Dual Gegenbauer interlacing]\label{dualinterlG} For  $x \in[-1,0[ \cup  ]0,1]$,  $G_n(x,z)$ interlaces $G_{n-1}(x,z)$ and $\partial_xG_n(x,z)$ in $z$.  At $x=0$, one of the two polynomials is identically zero, so depending on the conventions we could still say that interlacing happens. 
\end{theorem} 
\begin{remark}
The same properties for the Laguerre family can be deduced from the orthogonality of Charlier polynomials. 
\end{remark}

\begin{corollary}\label{dualinterlGmodified} For  $x \in[-1,0[ \cup  ]0,1]$,  $\hat{G}_n(x,z)$ interlaces $\hat{G}_{n-1}(x,z)$ and $\partial_x\hat{G}_n(x,z)$ in $z$.  Again, we can extend it to $x=0$ with the right conventions. 
\end{corollary}

It then follows, grouping results, that  $P_n(x,z)$   are real-rooted in $z$ for $x$ in the support of the underlying measure of orthogonality. This seems like a sharp result, in the sense that numerically we can find large degree polynomials with non-real roots in $z$ if $x$ is outside these intervals of orthogonality.

We show in Section~3 how to use such results to analyze $P_n(x,z)$ as bivariate polynomials and consider new derived families of polynomials. More specifically, we prove that when we differentiate these orthogonal polynomials with respect to $z$ and consider them as polynomials in $x$, then they are realrooted in $x$. Such families of polynomials (derivatives with respect to the parameter) seem to have many nice properties similar to their corresponding orthogonal polynomials from which they are derived and yet have never been studied.\begin{theorem}\label{LaguerreD} Consider $z> -1$.
The polynomials $\partial_{z}^kL_n(x,z)$ for all $k \leq n$ are real-rooted in $x$ of degree $n-k$. Additionally  their roots are  increasing with $z$ and they form a strict interlacing family of decreasing degree, in the sense that $\partial_{z}^{k+1}L_n(x,z)$ interlaces strictly $\partial_{z}^kL_n(x,z)$ in $x$. 
\end{theorem}

The analogue holds for Gegenbauer families:
\begin{theorem}\label{GegenbauerD}Consider  $z >-1/2$.
The polynomials  $\partial_{z}^k\hat{G}_n(x,z)$ for all $k \leq n$ are real-rooted in $x$ of degree $n$. The polynomials $\partial_{z}^{k}\hat{G}_n(x,z)$ have $\lfloor n/2 \rfloor$ positive roots for $k \leq n- \lfloor n/2 \rfloor$, and for larger $k$, $n-k $ positive roots (for $k \geq  n- \lfloor n/2 \rfloor$, zero roots add two by two when $k$ increases by one).
Additionally  the positive roots are  decreasing with $z$,  their symmetric negative counterpart  increasing with $z$, and the positive roots of $\partial_{z}^{k+1}\hat{G}_n(x,z)$ interlace strictly the positive roots of $\partial_{z}^k\hat{G}_n(x,z)$ (the same holds for the negative ones by symmetry). 
\end{theorem}
\begin{remark}
One can wonder why we chose $\hat{G}_n$ and not $G_n$. In fact we can adapt the proof of the above result to $G_n$ but only for $z>0$. Indeed the factor term $z$ that we get rid of when  considering modified Gegenbauer polynomials can become negative and it disrupts parts of the argument. Somehow, the modified Gegenbauer polynomials seem more adequate as bivariate polynomials, as they also have simple roots in $z$. 

 \end{remark}

\begin{remark}
 Let us remark that taking $k \leq n$ is not restrictive, as  $\partial_z^kL_n(x,z)$ and  $\partial_z^k\hat{G}_n(x,z)$ are identically zero if $k>n$. 
\end{remark}

We can wonder if such a phenomenon is  general and applies to other families of hypergeometric functions depending on a parameter. It seems plausible, as the proofs rely on very general monotonicity properties of the polynomials. We should also point out that this bivariate point of view on such polynomials came up while studying a continuous extension of polynomial convolutions considered in \cite{GM} themselves extending the ones from \cite{MSS}. Gegenbauer and generalized Laguerre polynomials are extremal polynomials in finite free probability (see Proposition 7.8 in \cite{G} for the link between rectangular finite free probability and Laguerre polynomials and Corollary 5.6 in \cite{GM} for the link with Gegenbauer polynomials), and differentiating with respect to $z$ has a meaning in the framework of finite free probability. But we chose to keep such a connexion for a future paper, as  these results themselves are relevant in the theory of orthogonal polynomials independently. 

\section{Realrootedness in $z$ of Gegenbauer polynomials}
The goal of this section is to prove Theorem~\ref{Gegenbauerz} and its corollaries. The following strategy could as well apply to Laguerre polynomials, but  it leads to already well known results. 
  
\subsection{Set up and Initialization: monotonicity and realrootedness at the extreme points of orthogonality}

\subsubsection{Set up and notations}
Let us treat separately the case $x=0$. Using the expression given in \ref{defnfamily}, we get:
\begin{align*}
G_n(0,z)&=\frac{(-1)^{n/2}}{(n/2)!} \prod_{j=0}^{n/2-1}(z+j) \text{   \;  when $n$ is even and}   &     G_n(0,z)&=0 \text{    \; when $n$ is odd}.
\end{align*}

Let us recall the explicit expansion of our family of polynomials at the extreme points using the hypergeometric form:

\begin{align*}
G_n^{(z)}(x)&:= \frac{\prod_{l=0}^{n-1} (2z+l)}{n!}  \vphantom{F}_2 F_1 (-n,2z+n;z+1/2; (1-x)/2)\\
&= \prod_{l=0}^{n-1} (2z+l) \sum_{k=0}^n (-1)^k \frac{1}{k!(n-k)!}\frac{ \prod_{i=0}^{k-1}(2z+n+i)}{ \prod_{i=0}^{k-1}(z+1/2+i)  2^k}(1-x)^k \\
&=\sum_{k=0}^n (-1)^k \frac{{2^{n} \binom{n}{k}}}{n!}\frac{ \prod_{i=0}^{n+k-1}(z+i/2)}{ \prod_{i=0}^{k-1}(z+(2i+1)/2) }(1-x)^k \\
&= \sum_{k=0}^n \frac{{2^{n} \binom{n}{k}}}{n!} \prod_{i=2k}^{n+k-1}(z+i/2) \prod_{i=0}^{k-1}(z+i) (x-1)^k  .
\end{align*}

As the Gegenbauer polynomials are even or odd in $x$, that is   $G_n(-x,z)= (-1)^nG_n(x,z)$, we will restrict to proving our results for $x\in[-1,0]$, and extend them afterwards by symmetry. 
We have:
\begin{equation} \label{decomposition}
G_n(x,z)=(-1)^n G_n(-x,z)= \sum_{k=0}^n\frac{{2^{n} \binom{n}{k}}}{n!} (-1)^{n+k}\prod_{i=2k}^{n+k-1}(z+i/2) \prod_{i=0}^{k-1}(z+i) (x+1)^k .
\end{equation}
So that $G_n(-1,z)=  (-1)^{n}  \frac{2^{n}}{n!}\prod_{i=0}^{n-1}(z+i/2)$, which is clearly realrooted in $z$ with simple roots.

Notice that we can write, if $n$ is even,
\[
G_n(x,z)=\Big[ \prod_{j=0}^{n/2-1}(z+j) \Big]\tilde{G_n}(x,z)
\]
and if $n$ is odd,
\[
G_n(x,z)=\Big[ \prod_{j=0}^{(n-1)/2}(z+j)\Big] \tilde{G_n}(x,z),
\]
where $\tilde{G_n}(x,z)$ is a bivariate polynomial. That is, approximately half of the roots in $z$, depending on the oddness, are constant when $x$ is moving. We can therefore investigate instead the evolution of the roots of $ \tilde{G_n}(x,z)$ to avoid considering roots that remain constant (and real). We can find the leading coefficient of the polynomial in $z$ as a function of $x$ by inspection of the formula in \ref{defnfamily} and we get $\frac{(2x)^{n}}{n!}$. 
\begin{definition}\label{defngtilde}
For $x \neq 0$, we can define a bivariate polynomial $\tilde{G_n}(x,z)$ of degree $n$ in $x$ and of degree $\lfloor n/2 \rfloor$ in $z$ as:
\begin{align*}
\tilde{G_n}(x,z)&=\frac{G_n(x,z)}{ \Big[ \prod_{j=0}^{\lceil n/2\rceil-1}(z-\mu_j) \Big] } \\
&=\frac{(2x)^{n}}{n!} \prod_{i=1}^{\lfloor n/2 \rfloor}\big(z-\gamma_i^n(x)\big)
\end{align*}
where $\mu_j= -j$ and the $\gamma_i^n(x)$ are a priori complex roots (defined only for $ x \neq 0$),where we order the roots by nondecreasing module: $ |\gamma_1^n(x)| \leq   |\gamma_2^n(x)| \leq \dots \leq |\gamma^n_{\lfloor n/2 \rfloor}(x)| $. 
\end{definition}

 It simplifies greatly for $x=-1$:
\[
\tilde{G_n}(x,z)_{|_{x=-1}}= \frac{ 2^{n}}{n!} (-1)^n  \prod_{i=1}^{\lfloor n/2\rfloor}(z+1/2+i-1) 
\]

so that  $\gamma_i(-1)= -1/2- (i-1)$ for $i=1,2\dots\lfloor n/2 \rfloor$ are all real and distinct. 

\subsection{Initialization}
We  derive a way to locally prove the existence of real-rootedness if it is known at some extreme point of the interval.

\begin{lemma}[Local existence of the roots and smoothness] \label{implicit}Take $a \in \mathbb{R}$. Assume $P(x,z)$ is a bivariate polynomial such that $P(a,z)$ is realrooted in $z$ of degree $j$ and has only simple roots in $z$ (let us call them $z_i(a)$ for $i=1\dots.j$). Assume that $P(x,z)$  has degree  $j$ for all $x$.Then for $x$ in a neighborhood of $a$, $P(x,z)$ is also realrooted in $z$ of degree $j$ with simple roots $z_i(x)$. Furthermore, the roots $z_i(x)$  are $C^{\infty}$ functions of $x$. 
\end{lemma}
\begin{proof}
 Consider the equation $P(x,z)=0$, around the points $\big(a,z_i(a)\big)$. \\
 We have  $\partial_zP(x,z)_{|_{x=a, z=z_i(a)}}  \neq 0$ as the roots are simple (they can't be roots of the derivative in $z$). Then using the implicit function theorem, we can find in a neighborhood of each point $z_i(a)$  a $C^{\infty}$ function  $z_i(x)$ which will be the only solution  to the equation $P(x,z)=0$ in this neighborhood. Therefore we have found $j$ roots, and it is the maximal number of roots for a fixed $x$, as the polynomial is of degree $j$. 
 \end{proof}
 
\begin{lemma} 
 $\tilde{G_n}(x,z)$ is real rooted of degree  $\lfloor n/2\rfloor$  in $z$ with simple roots in a neighborhood of $x=-1$.
\end{lemma}
\begin{proof}
We have:

 \[
 \tilde{G_n}(x,z)_{|_{x=-1}}= \frac{ 2^{n}}{n!} (-1)^n  \prod_{i=1}^{\lfloor n/2\rfloor}(z+1/2+i-1) 
\]
which has $\lceil n/2\rceil$  simple roots in $z$, so we just have to apply Lemma~\ref{implicit} with $a=-1$. 
\end{proof}

\begin{lemma}\label{local}
 The roots $\gamma_i(x)$ of $\tilde{G_n}(x,z)$ are all  increasing when $x$ is, in a neighborhood of $x=-1$. More precisely,  $\frac{d \gamma_i(x)}{dx}> 0$ for $x>-1$ in this neighborhood. 
\end{lemma}
\begin{proof}
We will prove this fact by first proving the following.

\begin{lemma}\label{derivdeepG}
We have: $ \frac{d^l \gamma_i(x)}{dx^l}_{|_{x=-1}}= 0 $  for $1\leq l < i$,   and    $\frac{d^i \gamma_i(x)}{dx^i}_{|_{x=-1}}> 0$.
\end{lemma} 

\begin{proof}
We consider $l \leq \lfloor n/2\rfloor$ in the following.  
Using Equation~\ref{decomposition} we get that
\[
\partial_x^lG_n(x,z)_{|_{x=-1}}=  l!\frac{{2^{n} \binom{n}{l}}}{n!}  (-1)^{n+l} \prod_{j=2l}^{n+l-1}(z+j/2) \prod_{j=0}^{l-1}(z+j) 
\]
so that by dividing by $\prod_{j=0}^{\lceil n/2\rceil-1}(z+j)$ and reorganizing the factors,
\[
\partial_x^l\tilde{G_n}(x,z)_{|_{x=-1}}=\frac{{2^{n} \binom{n}{l}} l!}{n!} (-1)^{n+l}  \prod_{j=l}^{\lceil n/2\rceil-2}(z+1/2+j)  \prod_{j=2\lceil n/2\rceil-1 }^{n+l-1}(z+j/2) .
\]

 We then see, using  that $\gamma_i(-1)= -1/2- (i-1)$ and $i \leq \lfloor n/2\rfloor$, that
\[
\partial_x^l\tilde{G_n}(x,z) \big(-1,\gamma_i(-1)\big)= 0   \text{      for $i \geq l+1$}  
\]
and
\[
(-1)^{n+i}\partial_x^i\tilde{G_n} \big(-1,\gamma_i(-1)\big)=\frac{{2^{n} \binom{n}{i}} i!}{n!}  \prod_{j=i}^{\lceil n/2\rceil-2}\big(j-(i-1)\big)  \prod_{j=2\lceil n/2\rceil-1}^{n+i-1}\big((j-1)/2- (i-1)\big) >0.
\]

Now we have $\tilde{G_n}\big(x,\gamma_i(x)\big)=0$ for all $i$, by definition, so differentiating with respect to $x$, we get:
\[
\frac{d \gamma_i(x)}{dx}= - \frac{\partial_x\tilde{G_n}}{\partial_z\tilde{G_n}} \big(x,\gamma_i(x)\big).
\]
 Note that the denominator is nonzero at $-1$ as the roots in $z$ are simple ( so they won't be roots of the derivative in $z$). Using induction on $l$, we get for $i > l\geq 1$, that
\[
\frac{d^l \gamma_i(x)}{dx^l}_{|_{x=-1}}= 0.
\]
Indeed when we differentiate $l-1$ times in $x$ the complicated expression $- \frac{\partial_x\tilde{G_n}(x,\gamma_i(x))}{\partial_z\tilde{G_n}(x,\gamma_i(x))}$, we will get a sum of  terms that all contain either $ \partial_x^k\tilde{G_n}(x,\gamma_i(x))_{x=-1}$ for $k \leq l$ (which are zero by the above computations), or  $\frac{d^k \gamma_i(x)}{dx^l}_{|_{x=-1}}$ for $k<l$ (which are zero by induction). For $l=i$, the only nonzero term comes from differentiating only the numerator with respect to the first variable $i$ times: 
\[
\frac{d^i \gamma_i(x)}{dx^i}_{|_{x=-1}}= - \frac{\partial_x^i\tilde{G_n}}{\partial_z\tilde{G_n}} \big(-1,\gamma_i(-1)\big).
\]

Now, we have $ \tilde{G_n}(x,z)_{|_{x=-1}}= (-1)^n \frac{2^{n}}{n!} \prod_{j=0}^{\lfloor n/2\rfloor-1}(z+1/2+j)$, so that
\begin{align*}
\partial_z \tilde{G_n}(x,z)\big(-1,\gamma_i(-1)\big)&=  (-1)^n \frac{2^{n}}{n!} \prod_{j=0, j \neq i-1}^{\lfloor n/2\rfloor-1}\big(j-(i-1)\big)\\
&=  (-1)^n \frac{2^{n}}{n!} (-1)^{i-1}\prod_{j=0}^{i-2}\big((i-1)- j \big) \prod_{j=i}^{\lfloor n/2\rfloor-1}\big(j-(i-1)\big) 
\end{align*}
and it follows that $(-1)^{n+i-1}\partial_z \tilde{G_n}(x,z)\big(-1,\gamma_i(-1)\big) >0$. Therefore
\[
\frac{d^i \gamma_i(x)}{dx^i}_{|_{x=-1}} >0
\]
as claimed. 
\end{proof}

Then Lemma~\ref{local} follows easily by Taylor expansion of the roots in $x$ around $-1$: 
\[
\gamma_i(x)= \gamma_i(-1)+ \frac{(x+1)^i}{i!}\frac{d^i \gamma_i(x)}{dx^i}_{|_{x=-1}} + o( (x+1)^i).
\]
\end{proof}

In the following we prove by induction the following statement: for $x \in [-1,0[ $, $\tilde{G_n}(x,z)$ is realrooted in $z$ with roots $\gamma^n_i(x)$ increasing to $+\infty$ when x converges to $0$. The case $n=1$ is trivial as $\tilde{G_n}(x,z)=1$. Assume that it is true for $\tilde{G}_{n-1}(x,z)$. 

\subsection{Interlacing properties: invariants preserved along realrootedness}
In this subsection, we find invariants that enable us to characterize realrootedness. We recall here that when there is no ambiguity,  $\gamma_i(x)$ refers to $\gamma^n_i(x)$. 

\begin{lemma}\label{simpleroots}  Assume $\gamma_i(x)$ are real and distinct for $x \in ]-1,b[=I$, $b \leq 0$, and all $i=1,2\dots \lfloor n/2 \rfloor$. Then $\partial_x\tilde{G_n}(x,\gamma_i(x))$  can't be zero for $x\in I$. As a result it has a constant sign on this interval. Equivalently,  $\tilde{G}_{n-1}(x,\gamma_i(x))$ can't be zero either. Said otherwise, we can't have a  shared root for  $\tilde{G}_{n-1} (x,z)$ or   $\partial_x \tilde{G}_n(x,z)$ and   $\tilde{G}_{n} (x,z)$ on $I$. 
\end{lemma}
\begin{proof}
 Note that the property is not true if we replace $\tilde{G_n}$ by $G_n$ as there can be trivial roots in z at the negative integers shared by the polynomial and its derivative, by inspection of the equality $\partial_x G_n(x,\gamma_i(x))= \prod_{j=0}^{\lceil n/2\rceil-1 }\big(\gamma_i(x)+j\big) \partial_x\tilde{G_n}\big(x,\gamma_i(x)\big)$. All the goal will be to bring ourselves back to positive parameters for which the results we claim are well known.
 
Let us assume by contradiction that $\partial_x \tilde{G_n} (x_0,\gamma_i(x_0)) =0$ some $i$ and $x_0 \in ]-1,b[$.  As  $\partial_x  \tilde{G_n} (x,\gamma_i(x)) $ is nonzero in a neighborhood of $x=-1$, $x>-1$ (see Lemma~\ref{derivdeepG}), then we can assume $x_0$ is the smallest $x>-1$ such that  $\partial_x \tilde{ G}_n (x,\gamma_i(x)) =0$.  The function $\partial_x \tilde{G_n} (x,\gamma_i(x))$ has a constant sign on $]-1,x_0]$, and as a consequence $\gamma_i(x)$ is  increasing in $x$ on this interval. 
As  $\gamma_i(-1)\geq - (\lfloor n/2 \rfloor-1)-1/2$ (which is the value of the extremal negative root), $\gamma_i(x_0)> - (\lfloor n/2 \rfloor-1/2) $.
Consider  the algebraic identity (see Equations~(4.7.27) in \cite{Szego}), valid for all $z$:
\[
(1-x_0^2) \partial_x G_n (x_0,z)
= -nx_0 G_n (x_0,z) +(n+2 z-1)G_{n-1} (x_0,z),
\] 
so that plugging in  $G_n (x_0,z)= \prod_{j=0}^{\lceil n/2\rceil-1 }\big(z+j\big) \partial_x\tilde{G_n}\big(x,z\big)$, $\partial_x G_n(x,z)= \prod_{j=0}^{\lceil n/2\rceil-1 }\big(z+j\big) \partial_x\tilde{G_n}\big(x,z\big)$ and  $G_{n-1}(x_0,z)= \prod_{j=0}^{\lceil (n-1)/2\rceil-1 }\big(z+j\big) \tilde{G}_{n-1}\big(x,z\big)$ and dividing by the common factor \\
$\prod_{j=0}^{\lceil (n-1)/2\rceil-1 }\big(z+j\big) $ on both sides we get:
\[
(1-x_0^2) \alpha(z)\partial_x \tilde{G}_n (x_0,z)
= -nx_0\alpha(z) \tilde{G}_n (x_0,z) +(n+2 z-1)\tilde{G}_{n-1} (x_0,z),
\] 
where $\alpha(z)=1$ or $\alpha(z)=z+\lceil n/2\rceil-1$ depending on the parity of $n$.
Now choose $z=\gamma_i(x_0)$ and remark that  $(n+2 \gamma_i(x_0)-1)>0$ as  $\gamma_i(x_0)> - (\lfloor n/2 \rfloor-1/2)$. On the other hand,   $\partial_x \tilde{G_n} (x_0,\gamma_i(x_0)) =0$  as well as $ \tilde{G_n} (x_0,\gamma_i(x_0)) =0$, and therefore we get $ \tilde{G}_{n-1} (x_0,\gamma_i(x_0)) =0$. Notice that choosing $x_0$ to be the smallest number such that $ \tilde{G}_{n-1} (x_0,\gamma_i(x_0)) =0$ is equivalent to defining it as the smallest number such that  $\partial_x  \tilde{G}_{n} (x_0,\gamma_i(x_0)) =0$, as the roots will be monotonous on $[-1,x_0]$ and the factors $\alpha(\gamma_i(x_0))$ and $(n+2 \gamma_i(x_0)-1)$ are nonzero.\\
 Then, using the recurrence relation (see Equations~(4.7.27) in \cite{Szego} again):
\[
\frac{n+1}{2n+2z}  G_{n+1} (x_0,z) = x_0 G_n (x_0,z) - \frac{n+2z-1}{2n+2z}G_{n-1} (x_0,z),
\]
and dividing by the common shared trivial factors as above we get:
\[
\frac{n+1}{2n+2z} (z+\lceil (n+1)/2\rceil-1)  \tilde{G}_{n+1} (x_0,z) = x_0 \alpha(z)\tilde{G}_n (x_0,z) - \frac{n+2z-1}{2n+2z}\tilde{G}_{n-1} (x_0,z).
\]
 As $-\lceil (n+1)/2\rceil+1 < -\lfloor n/2 \rfloor+1/2$, we have $\gamma_i(x_0) + \lceil (n+1)/2\rceil-1>0$, as well as  $2n+2\gamma_i(x_0)>0$, and if we plug in $z=\gamma_i(x_0)$ above we get  $\tilde{G}_{n+1} (x_0,\gamma_i(x_0))=0$. By induction, we prove exactly by the same method that  $\tilde{G}_{n+k} (x_0,\gamma_i(x_0))=0$ for all $k \in \mathbb{N}$.  
 If we go back to our first equation applied to $\tilde{G}_{n+k}$ instead of $\tilde{G}_{n}$ we obtain:
 \[
(1-x_0^2) \alpha_k(z)\partial_x \tilde{G}_{n+k} (x_0,z)
= -(n+k)x_0\alpha_k(z) \tilde{G}_{n+k} (x_0,z) +(n+k+2z-1)\tilde{G}_{n+k-1} (x_0,z),
\] 
 where  $\alpha_k(z)=1$ or $\alpha_k(z)=z+\lceil (n+k)/2\rceil-1$ depending on the parity of $n$. For $k \geq 1$,  $\lceil (n+k)/2\rceil-1 >\lfloor n/2 \rfloor- 1/2$ and  $\alpha_k(\gamma_i(x_0))>0$ so that plugging in $z=\gamma_i(x_0)$ above we get $\partial_x \tilde{G}_{n+k} (x_0,\gamma_i(x_0))=0$ (which is also true by assumption for $k=0$). 
 But we have for $|x_0|<1$ using Equations~(4.7.27) in \cite{Szego}:
\[
\partial_x G_{n+k} (x_0,z) = 2z G_{n+k-1} (x_0,z+1),
\]
so that simplifying by the common factors in $z$, we are left with:
\[
\partial_x \tilde{G}_{n+k} (x_0,z) = 2 \beta_k(z) \tilde{G}_{n+k-1}(x_0,z+1),
\]
where $\beta_k(z)= z+\lceil (n+k-1)/2\rceil$ or  $\beta_k(z)=1$ according to the parity of $n$. We plug in $z=\gamma_i(x_0)$ and as $\beta_k(\gamma_i(x_0))>0$, we readily get $\tilde{G}_{n+k-1}(x_0,\gamma_i(x_0)+1)=0$, for all $k \in \mathbb{N}$.  Then it is easy to show by induction, using the above strategy, replacing  $\gamma_i(x_0)$ with $\gamma_i(x_0)+j-1$,  that   $\tilde{G}_{n+k} (x_0,\gamma_i(x_0)+j)=0$ for all  $j \in \mathbb{N}$, and all $k \in \mathbb{N}$. For $j$ larger than $\lfloor n/2 \rfloor$, the parameter $\gamma_i(x_0)+ j$ is positive, and we can now use classical results on roots of Gegenbauer polynomials. This means that successive Gegenbauer polynomials with parameter  $\gamma_i(x_0)+j$ have the root $x_0$ in common, so that their derivatives share these roots too,  which is absurd as their roots are simple by orthogonality. We conclude that $\partial_x \tilde{G}_n (x,\gamma_i(x))$ has a constant sign for all  $x\in]-1,b[$. Note that the above proof allows us to say that  $\tilde{G}_{n-1} (x,\gamma_i(x))$  can't be zero either.

\end{proof}
\begin {remark} \label{extremeequality}
In the case in which $lim_{x \to b} \gamma_i(x)= \mu$ is finite, we have by continuity $\tilde{G}_{n} (b,\mu)=0$. Then using exactly the same reasoning as above we can prove that we can't have $ \tilde{G}_{n-1} (b,\mu)=0$. Indeed, we will still have $\mu> - (\lfloor n/2 \rfloor-1/2)$ and the whole proof can be applied exactly the same way. 
\end{remark}

\begin{corollary}\label{extendedmonotonicitygegenbauer} Assume $\gamma_i(x)$ are real and distinct for all $i$  and $x \in ]-1,b[$, $b<0$. Then it follows from the previous proof that for  $x \in ]-1,b[$,
\[
 \frac{d\gamma_i(x)}{dx}>0.
 \]
\end{corollary}

\begin{lemma}\label{interlacing} Consider an interval $I=[-1,b[$ such  that  $\tilde{G_n}(x,z)$ has simple real roots in $z$ on $I$, then by induction the same will be true of $\tilde{G}_{n-1}(x,z)$ and their roots interlace strictly on$ ]-1,b[$ (they have common roots at $-1$). 
\end{lemma}
\begin{proof}
Let us write 
\begin{align*}
\tilde{G_n}(x,z)&=\frac{ (2x)^{n}}{n!}   \prod_{i=1}^{\lfloor n/2\rfloor}\big(z- \gamma_i^n(x) \big)   &\text{and}&& \tilde{G}_{n-1}(x,z)&= \frac{ (2x)^{n-1}}{(n-1)!}   \prod_{i=1}^{\lfloor (n-1)/2\rfloor}\big(z- \gamma_i^{n-1}(x) \big).
\end{align*}
Notice that if $n$ is odd, the  polynomials $\tilde{G}_n(x,z)$ and  $\tilde{G}_{n-1}(x,z)$ are of the same degree, but that for $n$ even, $\tilde{G}_n(x,z)$ has one more root. If n is odd, we will show that for all $x \in I$,  all $i \leq  (n-1)/2-1$,  $\gamma_i^n(x) > \gamma_i^{n-1}(x) > \gamma_{i+1}^n(x)$, and $ \gamma_{(n-1)/2}^n(x) > \gamma_{(n-1)/2}^{n-1}(x) $ . If $n$ is even, we will show that for all $x \in I$,  all $i \leq  n/2-1$,  $\gamma_i^n(x) > \gamma_i^{n-1}(x) > \gamma_{i+1}^n(x)$.  \\

We first check the property locally, that is in a neighborhood of $-1$, above $-1$.
We have (see computations from Lemma~\ref{derivdeepG}):
\begin{align*}
\frac{d^i \gamma_i^n(x)}{dx^i}_{|_{x=-1}}&= - \frac{\partial_x^i\tilde{G_n}}{\partial_z\tilde{G_n}} \big(-1,\gamma_i(-1)\big)\\
& =  - \frac{ (-1)^{n+i}  \frac{{2^{n} \binom{n}{i}} i!}{n!}\prod_{j=i}^{\lceil n/2\rceil-2}\big(j-(i-1)\big)  \prod_{j=2\lceil n/2\rceil-1}^{n+i-1}\big((j-1)/2- (i-1)\big)}{ \frac{2^{n}}{n!} (-1)^{n+i-1}\prod_{j=0}^{i-2}\big((i-1)- j \big)\prod_{j=i}^{\lfloor n/2\rfloor-1}\big(j-(i-1)\big) }.\\
\end{align*}
Assume $n$ is even, then:
\begin{align*}
\frac{d^i \gamma_i^n(x)}{dx^i}_{|_{x=-1}}&=  \frac{ \binom{n}{i} i!\prod_{j=i}^{ n/2-2}\big(j-(i-1)\big)  \prod_{j=n-1}^{n+i-1}\big((j-1)/2- (i-1)\big)}{ \prod_{j=0}^{i-2}\big((i-1)- j \big)\prod_{j=i}^{n/2-1}\big(j-(i-1)\big) },\\
\frac{d^i \gamma_i^{n-1}(x)}{dx^i}_{|_{x=-1}}&= \frac{ \binom{n-1}{i} i!\prod_{j=i}^{ n/2-2}\big(j-(i-1)\big)  \prod_{j=n-1}^{n+i-2}\big((j-1)/2- (i-1)\big)}{ \prod_{j=0}^{i-2}\big((i-1)- j \big)\prod_{j=i}^{n/2-2}\big(j-(i-1)\big) }.
\end{align*}
So that:
\begin{align*}
\frac{d^i \gamma_i^n(x)}{dx^i}_{|_{x=-1}}&= \frac{d^i \gamma_i^{n-1}(x)}{dx^i}_{|_{x=-1}} \frac{ \binom{n}{i} } { \binom{n-1}{i}}  \frac{\big((n+i-2)/2 - (i-1)\big)}{\big(n/2-1 - (i-1)\big)}.
\end{align*}
Similarly, if $n$ is odd, then:
\begin{align*}
\frac{d^i \gamma_i^n(x)}{dx^i}_{|_{x=-1}}&=  \frac{ \binom{n}{i} i!\prod_{j=i}^{ (n+1)/2-2}\big(j-(i-1)\big)  \prod_{j=n}^{n+i-1}\big((j-1)/2- (i-1)\big)}{ \prod_{j=0}^{i-2}\big((i-1)- j \big)\prod_{j=i}^{(n-1)/2-1}\big(j-(i-1)\big) },\\
\frac{d^i \gamma_i^{n-1}(x)}{dx^i}_{|_{x=-1}}&= \frac{ \binom{n-1}{i} i!\prod_{j=i}^{ (n-1)/2-2}\big(j-(i-1)\big)  \prod_{j=n-2}^{n+i-2}\big((j-1)/2- (i-1)\big)}{ \prod_{j=0}^{i-2}\big((i-1)- j \big)\prod_{j=i}^{(n-1)/2-1}\big(j-(i-1)\big) }.\\
\end{align*}
So that again:
\begin{align*}
\frac{d^i \gamma_i^n(x)}{dx^i}_{|_{x=-1}}&= \frac{d^i \gamma_i^{n-1}(x)}{dx^i}_{|_{x=-1}} \frac{ \binom{n}{i} } { \binom{n-1}{i}}  \frac{ \big((n+i-2)/2 - (i-1)\big)}{ \big(n-2)/2-(i-1)\big) }.
\end{align*}

As $\frac{ \binom{n}{i} } { \binom{n-1}{i}}  \frac{ \big((n+i-2)/2 - (i-1)\big)}{ \big(n-2)/2-(i-1)\big) } >1$, we conclude that for all $ 1\leq i \leq \lfloor (n-1)/2\rfloor$:

\[
\frac{d^i \gamma_i^n(x)}{dx^i}_{|_{x=-1}} >  \frac{d^i \gamma_i^{n-1}(x)}{dx^i}_{|_{x=-1}}. 
\]
As   $\gamma_i^n(-1)=    \gamma_i^{n-1}(-1)=  \gamma_i(-1)=  -1/2- (i-1)$, we can do a Taylor expansion around $x=-1$:
\[
\gamma_i^n(x)= \gamma_i(-1)+ \frac{(x+1)^i}{i!}\frac{d^i \gamma_i^n(x)}{dx^i}_{|_{x=-1}} + o( (x+1)^i),
\]
\[
  \gamma_i^{n-1}(x)= \gamma_i(-1)+ \frac{(x+1)^i}{i!}\frac{d^i \gamma_i^{n-1}(x)}{dx^i}_{|_{x=-1}} + o( (x+1)^i).
  \]
For  $1\leq i \leq \lfloor (n-1)/2\rfloor$,  it is then clear that in a neighborhood of $-1$ and above $-1$, $\gamma_i^n(x) > \gamma_i^{n-1}(x) $. \\
In the case  $i \leq  (n-1)/2-1$ if $n$ is odd and  $i \leq n/2-1$ if $n$ is even, as $\gamma_i^{n-1}(-1) - \gamma_{i+1}^n(-1)=1$, we also get  $\gamma_i^{n-1}(x) > \gamma_{i+1}^n(x)$ in a neighborhood of $-1$.   \\
Now as for all $i$, $\gamma_i^n(x)$, $\gamma_i^{n-1}(x)$ and $\gamma_{i+1}^n(x)$ are continuous functions of $x$, if by contradiction such inequalities were to fail for an $x \in I$, then there would exist $x_0$ and $i$ such that $\gamma_i^n(x_0)= \gamma_i^{n-1}(x_0)$ or $\gamma_i^{n-1}(x_0)=  \gamma_{i+1}^n(x_0)$. But then  this would mean that $\gamma_i^{n-1}(x_0)$ is a root of $G_n(x_0,z)$ and  $G_{n-1}(x_0,z)$, which is impossible by  Lemma~\ref{simpleroots}. Therefore we conclude that the inequalities above 
hold for all $x \in I$.  

\end{proof}

\begin{lemma}\label{interlacing} Consider an interval $I=[-1,b[$ such  that  $\tilde{G}_n(x,z)$ has simple real roots in $z$ on $I$, then the same will be true of $\partial_x\tilde{G}_n(x,z)$ and the roots of the two polynomials interlace strictly on $]-1,b[$ .  
\end{lemma}
\begin{proof} We bring ourselves back to a variant of the previous theorem by using the equality valid for $|x|<1$ (see Equations~(4.7.27) in \cite{Szego}):
\[
\partial_x G_{n} (x,z) = 2zG_{n-1} (x,z+1).
\]
As \begin{align*}
\partial_x G_{n} (x,z) &= \Big[\prod_{j=0}^{\lceil n/2\rceil-1}(z+j) \Big]\partial_x \tilde{G}_{n} (x,z)   &  G_{n-1} (x,z+1) &= \Big[\prod_{j=0}^{\lceil (n-1)/2\rceil-1}(z+j+1) \Big] \tilde{G}_{n-1} (x,z+1),
\end{align*}
we get
\[
\partial_x \tilde{G}_{n} (x,z) = \frac{ \prod_{j=0}^{\lceil (n-1)/2\rceil-1}(z+j+1)}{ \prod_{j=0}^{\lceil n/2\rceil-1}(z+j)}2z \tilde{G}_{n-1} (x,z+1) 
\]
and
 \begin{align*}
\partial_x \tilde{G}_{n} (x,z) &= 2 (z+n/2)\tilde{G}_{n-1}(x,z+1) \text{\; \;if $n$ is even,}  & \partial_x \tilde{G}_{n} (x,z) &= 2 \tilde{G}_{n-1}(x,z+1)\text{\; \;  if $n$ is odd}.
\end{align*}

If $n$ is even,  $\tilde{G}_{n} (x,z)$ has $n/2$ roots, one more root than $\tilde{G}_{n-1}(x,z)$ in $z$, and we want to show that for all $i \leq n/2-1$, 
\begin{align*}
 \gamma_i^n(x) &> \gamma_i^{n-1}(x) -1> \gamma_{i+1}^n(x)      \text{\; \;\;\;and}   &  \gamma_{n/2}^n(x) > -n/2.
\end{align*}
If $n$ is odd  $\tilde{G}_{n} (x,z)$ has the same number of roots than $\tilde{G}_{n-1}(x,z)$, and we want to show that for all $i \leq (n-1)/2 -1$,
\begin{align*}
 \gamma_i^n(x) &> \gamma_i^{n-1}(x) -1> \gamma_{i+1}^n(x)       \text{\; \;\;\;and}  &   \gamma_{(n-1)/2}^n(x) > \gamma_{(n-1)/2}^{n-1}(x) -1.
\end{align*}

We have directly $\gamma_{n/2}^n(x)= 1/2-\lfloor n/2\rfloor >-n/2 $.

First we check the other inequalities in a neighborhood of $-1$, $x>-1$. Let us consider $i$ in the proper interval according to the parity of $n$. We can verify that the inequality $\gamma_i^n(x) > \gamma_i^{n-1}(x) -1$ is true in a neighborhood of $-1$ as $\gamma_i^n(-1) = \gamma_i^{n-1}(-1) $. The nontrivial inequalities are $\gamma_i^{n-1}(x) -1> \gamma_{i+1}^n(x)$ for $x>-1$. We have equality at $-1$ as $\gamma_i^n(-1)=\gamma_i^{n-1}(-1):= \gamma_i(-1)$ and    $\gamma_i(-1) -1= \gamma_{i+1}(-1)$. Then we look at the Taylor expansions around $x=-1$:
\[
\gamma_i^{n-1}(x)-1= \gamma_{i+1}(-1)+ \frac{(x+1)^i}{i!}\frac{d^i \gamma_i^{n-1}(x)}{dx^i}_{|_{x=-1}} + o( (x+1)^i),   
\]
\[
    \gamma_{i+1}^{n}(x)= \gamma_{i+1}(-1)+\frac{ (x+1)^{i+1}}{(i+1)!}\frac{d^{i+1} \gamma_{i+1}^{n}(x)}{dx^{i+1}}_{|_{x=-1}} + o( (x+1)^{i+1}).
\]
It is clear then that locally $\gamma_i^{n-1}(x)-1>  \gamma_{i+1}^{n}(x)$ as $(x+1)^{i+1} <<(x+1)^i$.\\
 We extend the inequalities to the whole interval $I$ by noticing again that if the inequalities were to fail, then there would have to exist an index $j$ and an $x$ by continuity such that $\gamma_j^n(x) = \gamma_j^{n-1}(x) -1$ or  $\gamma_j^{n-1}(x) -1= \gamma_{j+1}^n(x)$.  This would mean $\partial_x \tilde{G_{n}} (x,\gamma_j^{n}(x))=0$ and as $ \tilde{G}_{n} (x,\gamma_j^{n}(x))=0$, we would  get a contradiction by Lemma~\ref{simpleroots}.
Notice that we proved in the process that $\tilde{G}_{n-1}(x,z+1)$ and  $\tilde{G}_{n}(x,z)$ interlace strictly too. 

\end{proof}

\subsection{Ordinary differential equation on the roots and global extension}

\begin{lemma}The local Gegenbauer property is in fact true over the whole interval:  $\tilde{G_n}(x,z)$ is real rooted in $z$ with simple (distinct) roots for $x\in[-1,0[$ , and they are all increasing to $+\infty$ when $x \to 0$. 

\end{lemma}
\begin{proof}
Denote by $F_n(x,z):=  - \frac{\partial_x\tilde{G_n}}{\partial_z\tilde{G_n}} \big(x,z\big)$. Consider a  rectangular domain $D$ such that $\partial_z\tilde{G_n}(x,z)$ is nonzero on it. 
 $F_n$ is continuous in $x$  and $z$ in the the domain $D$.
Indeed, it is a rational fraction whose denominator is nonzero and it is therefore $C^{\infty}$ in both variables by theorem of composition.  As we saw before,  $\tilde{G_n}(-1,z)$ is realrooted in $z$ with simple roots, $\partial_z\tilde{G_n}(-1,\gamma_i(-1)) \neq 0$ and by continuity we can find  small rectangles $D_i:=  [-1,-1+\epsilon] \times [\gamma_i(-1), \gamma_i(-1)+ \delta]$ such that $\partial_z\tilde{G_n}(x,z)$ is nonzero on $D_i$. A strong version of Picard's theorem tells us that there are maximal intervals $I_i^{max}= ]-1,0[ \cap[-1, \eta^{max}_i[$ for which the roots $\gamma_i(x)$  ($i=1,2\dots\lfloor n/2 \rfloor$) are the unique solutions of the initial value ODE
\[
\frac{dz}{dx}= F_n(x,z),    \;  \; \; \; \; \; \;  \;  z(-1)= -1/2- (i-1).
\]
Note that on $I_i^{max}$, $\partial_z\tilde{G_n}(x,\gamma_i(x)) \neq 0$ (the denominator is nonzero, so that the differential equation is well defined). 

Let us consider the smallest interval $I^{max}=I_{i_0}^{max}=]-1,b[$, with $b=\eta^{max}_{i_0}=\inf _{i=1\dots n}\eta^{max}_i$, so that $I^{max} \subset I_i^{max}$ for all $i$, and let us prove that $b=0$.\\
By Corollary~\ref{extendedmonotonicitygegenbauer}, $F_n(x,\gamma_i(x))>0$ for all $i$ on $I_{max}$ as all roots are distinct (nonzero denominator) and real (by definition of the intervals $I^i_{max}$).
According to Picard's theorem,  we either have $ \gamma_{i_0}(x)  \rightarrow_{x \to b} +\infty$ (explosion necessarily to +$\infty$ as the roots are increasing), or $b$ is such that $ \lim_{x \to b} F_n( x,\gamma_{i_0}(x))$ is not well defined (we leave the domain of definition). 
Let us prove that  $I^{max}=[-1,0[$  and that there is explosion when $x \to 0$ (roots going to infinity).\\

Assume by contradiction that $b<0$. \\
 Explosion can't happen if  $b<0$. Indeed, write  $G_n(x,z)= \frac{(2x)^n}{n!}z^n +P_{n-1}(x)z^{n-1} +...  = \prod_{j=0}^{\lceil n/2\rceil-1}(z-\mu_j) \frac{(2x)^{n}}{n!} \prod_{i=1}^{\lfloor n/2 \rfloor}\big(z-\gamma_i(x)\big) $  where $P_{n-1}(x)$ is a polynomial of degree $n-1$ as well as the coefficient of $z^{n-1}$ in the expansion of $G_n(x,z)$. We then get by identifying coefficients of $z^{n-1}$ on both sides:   
\[
\sum_i{\gamma_i(x)} + \sum_j{\mu_j} = \frac{-P_{n-1}(x)}{\frac{(2x)^{n}}{n!}}
\]
The right hand side is bounded, and as a result  $\gamma_{i_0}(x)$ can't diverge. \\

The alternative is that we leave the domain of definition. It happens if $\lim_{x \to b} \partial_z\tilde{G_n}  \big(x,\gamma_{i_0}(x)\big) =0$. We have seen that $\partial_z\tilde{G_n}(b,z)$ would be of degree  $\lfloor n/2 \rfloor-1 $ in $z$. By assumption, $\lim _{x \to b}\gamma_{i_0}(x)= \mu$ where $\mu$ is a root of  $\partial_z\tilde{G_n}(b,z)$. We can extend by continuity $\gamma_{i_0}(x)$ at $x=b$ with $\gamma_{i_0}(b) = \mu$. We check by continuity that $\tilde{G_n}\big(b,\gamma_{i_0}(b)\big)= \partial_z\tilde{G_n}\big(b,\gamma_{i_0}(b)\big )=0$ so that in fact $\gamma_{i_0}(b)$ is a real double root in $z$ of $\tilde{G_n} \big(b,z\big)$. Now it means that either $ \gamma_{i_0+1}(b) = \gamma_{i_0}(b)$ or $\gamma_{i_0-1}(b) = \gamma_{i_0}(b)$. Using the fact that for $x<b$, $\gamma_{i_0}(x) > \gamma_{i_0}(x) >\gamma_{i_0+1}(x) $, and then using Lemma~\ref{interlacing} (valid for $x<b$): $\gamma_{i_0-1}(x) >\gamma^{n-1}_{i_0-1}(x) > \gamma_{i_0}(x) >\gamma^{n-1}_{i_0}(x) > \gamma_{i_0+1}(x) $  \big(we recall that $\gamma_{i_0}(x)=\gamma^{n}_{i_0}(x)$\big), and then by continuity we obtain either $\gamma^{n-1}_{i_0-1}(b)= \gamma_{i_0}(b)$ or  $\gamma^{n-1}_{i_0}(b)= \gamma_{i_0}(b)$. Either way, using Remark~\label{extremeequality}, we would get a contradiction as we have a shared root for $\tilde{G_n}\big(b,z)$ and  $\tilde{G}_{n-1}\big(b,z)$. \\
 Therefore, we have necessarily $b=0$ and $\eta^{max}_i=0$ for all $i=1\dots\lfloor n/2 \rfloor$. \\

Let us prove that $\gamma_i(x)\to_{x \to 0} +\infty$ for all $i$. \\
Assume $n$ is odd. Then for $x<0$, $\tilde{G_n}\big(x,z)$ and $\tilde{G}_{n-1}\big(x,z)$ are both of degree $(n-1)/2$ and by the interlacing property from Lemma~\ref{interlacing}, $\gamma_{i}(x) >\gamma^{n-1}_{i}(x)$ for all $i \leq (n-1)/2$, and as by induction  $lim_{x \to 0}\gamma^{n-1}_{i}(x)= +\infty$, we get the result.\\
Now, if $n$ is even, for $x<0$, $\tilde{G_n}\big(x,z)$ is of degree $n/2$ and $\tilde{G}_{n-1}\big(x,z)$ of degree $n/2-1$, and consequently we can use the same argument as when $n$ is odd for all $i<n/2$ (roots will be lower bounded by roots of $\tilde{G}_{n-1}$), but we can't conclude for $\gamma_{n/2}(x)$. Hopefully, we have in this case $G_n(0,z)=\frac{(-1)^{n/2}}{(n/2)!} \prod_{j=0}^{n/2-1}(z+j)$ so that $\tilde{G}_n(0,z)=\frac{(-1)^{n/2}}{(n/2)!}$. Assume by contradiction that we don't have $lim_{x \to 0}\gamma_{n/2}(x)= +\infty$. As $\gamma_{n/2}(x)$ is monotonous for $x \in ]-1,0[$, then we have necessarily that  $\lim_{x \to 0}\gamma_{n/2}(x)= \mu$ exists and is finite. By continuity we have $\tilde{G_n}(0,\mu)=0$. Given the actual value (constant, nonzero), this is absurd.\\

\end{proof}

\subsection{Putting results together and extending them to modified Gegenbauer families}
Recall that  $G_n(x,z)= \prod_{j=0}^{\lceil n/2\rceil-1}(z+j) \tilde{G_n}(x,z)$, so it is clear that for $x \in[-1,0[$, $G_n(x,z)$ is realrooted in $z$ and as $\tilde{G_n}(x,z)$ only has simple roots by the above, there can be at most  double roots. As $\gamma_{i}(x)$  ranges continuously from $1/2-i$ to $+\infty$, there is a double root when it crosses  $-j$ for $0\leq j \leq \lceil n/2\rceil-1$. \\
For  $x \in]0,1]$,  we get from the equality $G_n(-x,z)= (-1)^nG_n(x,z)$, that $\gamma_{i}(-x)$ (for $i\leq \lfloor n/2\rfloor$) is a root of  $\tilde{G_n}(x,z)$ , and it is decreasing in $x$ by composition. Basically, the roots are the same, we are just reversing the interval. We can deduce right away from this  Theorem~\ref{Gegenbauerz}. We also get that for  $x \in[-1,0[ \cup  ]0,1]$,  $\tilde{G_n}(x,z)$ has $\lfloor n/2 \rfloor$ simple real roots in $z$ and interlaces  $\tilde{G}_{n-1}(x,z)$ and $\partial_x\tilde{G}_n(x,z)$ in $z$.  
To get  Theorem~\ref{dualinterlG}, we need to patch constant roots and check if interlacing still holds. We have: $\partial_xG_n(x,z)= \partial_x\tilde{G_n}(x,z) \prod_{j=0}^{\lceil n/2\rceil-1}(z+j)$. So the interlacing of $G_n(x,z)$ and $\partial_xG_n(x,z)$ follows directly as we just add the same set of roots to  $\tilde{G_n}(x,z)$ and  $\partial_x\tilde{G_n}(x,z)$. 
In the case of $G_{n-1}(x,z)$ if $n$ is even then we again add the same constant roots, which doesn't affect interlacing. If $n$ is odd, then there is an additional constant root $ -(n+1)/2+1$ for $G_{n}(x,z)$. On the other hand, $\tilde{G_n}(x,z)$ and $\tilde{G}_{n-1}(x,z)$  have the same number of roots and the lowest root in the interlacing is $\gamma^{n-1}_{(n-1)/2}(x)$. We just have to check that $\gamma^{n-1}_{ (n-1)/2}(x) \geq -(n+1)/2+1$. But this is clear as $\gamma^{n-1}_{(n-1)/2}(x)\geq \gamma^{n-1}_{(n-1)/2}(-1)= -n/2+1$. \\

Let us finally deduce results for the derived family $\hat{G}_n(x,z)$. We have : 
\[
\frac{\prod_{i=0}^{n-1}(z+1/2+i)}{\prod_{i=0}^{n-1}(z+i/2)}= \frac{\prod_{i=\lfloor n/2\rfloor}^{n-1}(z+1/2+i)}{\prod_{i=0}^{\lceil n/2\rceil-1}(z+i)}, 
\]

therefore:
 \[
\hat{G}_n(x,z)=  \frac{\prod_{i=0}^{n-1}(z+1/2+i)}{2^n\prod_{i=0}^{n-1}(z+i/2)} G_n^{(z)}(x)=\prod_{i=\lfloor n/2\rfloor}^{n-1}(z+1/2+i) \frac{\tilde{G_n}(x,z)}{2^n}.
\]
Notice that the number of constant roots in $z$ of $\hat{G}_n(x,z)$ is the same as for $G_n^{(z)}(x)$ (that is $\lceil n/2\rceil$), and the nonconstant roots are the same $\gamma_i(x)$, $i\leq \lfloor n/2\rfloor$, verifying: $\gamma_i(x) \geq  1/2-i \geq 1/2-\lfloor n/2\rfloor > -1/2-\lfloor n/2\rfloor$. They are above all the constant roots in  $ \prod_{i=\lfloor n/2\rfloor}^{n-1}(z+1/2+i)$, and consequently the roots are always simple. Corollary~\ref{Gegenbauerzmod} follows directly. \\
As for the interlacing properties of $\hat{G}_n(x,z)$, interlacing with $\partial_x\hat{G}_n(x,z)$ is again straightforward. In the case of $\hat{G}_{n}(x,z)$ and $\hat{G}_{n-1}(x,z)$, if $n$ is even then $\tilde{G}_{n-1}(x,z)$ and $\tilde{G}_{n}(x,z)$ have degrees $n/2-1$ and $n/2$ with lowest roots in the chain of interlacing verifying: $\gamma^{n}_{n/2}(x)< \gamma^{n-1}_{ n/2-1}(x)$. The constant root belonging to $\hat{G}_n(x,z)$ and not  $\hat{G}_{n-1}(x,z)$ is $-1/2-(n-1)$, and the constant root belonging to  $\hat{G}_{n-1}(x,z)$ and not  $\hat{G}_{n}(x,z)$ is $-1/2-(n/2-1)$. As we have $-1/2-(n-1)<-1/2-(n/2-1)<\gamma^{n}_{n/2}(x)< \gamma^{n-1}_{n/2-1}(x)$, interlacing still holds. If $n$ is odd,  then $\tilde{G}_{n-1}(x,z)$ and $\tilde{G}_{n}(x,z)$ both have degree $(n-1)/2$ and the lowest root in the interlacing chain is $\gamma^{n-1}_{ (n-1)/2}(x)$. The only constant root not shared by  $\hat{G}_n(x,z)$ and  $\hat{G}_{n-1}(x,z)$ belongs to $\hat{G}_n(x,z)$ and is equal to $ -1/2-(n-1)<\gamma^{n-1}_{ (n-1)/2}(x)$. As a result, the interlacing follows, and Corollary~\ref{dualinterlGmodified}.

\section{Applications to realrootedness in $x$ of derivatives in $z$}

The main application of realrootedness in $z$ will show up through Laguerre's inequality.
\begin{lemma}[Laguerre's inequality] \label{Laguerrein}
For a polynomial $p(z)$ with simple real roots, then 
\[
p''(z)p(z)-p'(z)^2 <0.
\]
If there are multiple roots, then we only have
\[
p''(z)p(z)-p'(z)^2  \leq 0.
\]
\end{lemma}

\subsection{Laguerre case}
The goal of this subsection is to prove Theorem~\ref{LaguerreD}. Consider in the following $z>-1$.
We show inductively the following property: for $0 \leq k \leq n$, $\partial_{z}^{k}L_n(x,z)$ is realrooted with distinct roots of degree $n-k$, with roots increasing in $z$. Furthermore, for $k\geq1$, the roots of $\partial_{z}^{k}L_n(x,z)$ interlace strictly the roots of $\partial_{z}^{k-1}L_n(x,z)$. 
The case $k=0$ is pretty clear using Theorem~\ref{monoroots}. Let us assume it is true up to $k$ by induction. 

\begin{lemma} \label{initialLaguerre}
The $n-(k+1)$ roots of  $\partial_z^{k+1}L_n(x,z)$ in $x$ are real  and interlace strictly those of $\partial_z^{k}L_n(x,z)$ in $x$.
\end{lemma}
\begin{proof}
We have
 \begin{align}
\partial_z^{k+1}L_n(x,z)= \sum_{i=0}^n \frac{(-1)^i \partial_z^{k+1}\big[ \prod_{j=i+1}^{n} (z+j)\big]}{i!(n-i)!} x^i.
\end{align}
 The factor $\big[ \prod_{j=i+1}^{n} (z+j)\big]$ is of degree $n-i$ so the largest $i$ such that differentiating $k+1$ times gives a nonzero term is $i= n-(k+1)$. Therefore the degree of  $\partial_z^{k+1}L_n(x,z)$ in $x$ is $n-(k+1)$. Consider the $n-k$ real roots $x_i(z)$ of $L_n(x,z)$ (that exist by induction).
 By differentiating $\partial_z^{k}L_n(x_i(z),z)=0$ with respect to $z$, we get
\begin{align}
 \partial_z^{k+1}L_n(x_i(z),z)= -\partial_x\partial_z^{k}L_n(x_i(z),z) \frac{d}{dz}x_i(z).
\end{align}
On the one hand, using induction, we have $\frac{d}{dz}x_i(z)> 0$. On the other hand, as roots $x_i(z)$ are distinct, the derivative $\partial_x\partial_z^{k}L_n(x_i(z),z) \neq 0$ changes sign when we increment $i$. We get that  $\partial_z^{k+1}L_n(x_i(z),z)$ changes sign $n-k-1$ times and consequently using the intermediate value theorem for $\partial_z^{k+1}L_n(x,z)$ we have $n-(k+1)$ real zeros strictly between zeros of $L_n(x,z)$, a fortiori all distinct.
In the end, all the zeros of   $\partial_z^{k+1}L_n(x,z)$ have been found and are interlacing strictly with the zeros of $\partial_z^{k}L_n(x,z)$.
\end{proof}

Now let us prove that the roots of  $\partial_z^{k+1}L_n(x,z)$ are also monotonous. Assume $k<n-1$, otherwise there is no root to consider. 
\begin{lemma}\label{initialrootLaguerre}
  The roots $\tilde{x}_i(z)$, $i=1\dots n-(k+1)$, of $\partial_z^{k+1}L_n(x,z)$ are such that for all $z$, $
\frac{d \tilde{x_i}}{dz} > 0$.
  \end{lemma}
  \begin{proof}
  As by strict interlacing $\partial_z^{k}L_n(\tilde{x_i}(z),z) \neq 0$, we can write:
\[
   \frac{\partial_z^{k+1}L_n}{\partial_z^{k}L_n}(\tilde{x_i}(z),z)=0,
\]
which leads to
\begin{align}
  \frac {d( \frac{\partial_z^{k+1}L_n}{\partial_z^{k}L_n}(\tilde{x_i}(z),z)\Big)}{dz} = \partial_x \Big(\frac{\partial_z^{k+1}L_n}{\partial_z^{k}L_n}(\tilde{x_i}(z),z)\Big)\frac{d \tilde{x_i}}{dz} + \partial_z \Big( \frac{\partial_z^{k+1}L_n}{\partial_z^{k}L_n}(\tilde{x_i}(z),z)\Big)= 0.
\end{align}
On the one hand, using fraction decomposition by noticing that $\partial_z^{k}L_n(x,z)$ has simple roots, and that the degree of $\partial_z^{k+1}L_n$ is one less than $\partial_z^{k}L_n$,
\[
\frac{\partial_z^{k+1}L_n}{\partial_z^{k}L_n}(x,z)=  \sum_{j=1}^{n-k} \frac {\partial_z^{k+1}L_n}{\partial_x\partial_z^{k}L_n}(x_j,z) \frac{1}{ x-x_j},
\]

so that:
\begin{align}
\partial_x\Big(\frac{\partial_z^{k+1}L_n}{\partial_z^{k}L_n}(x,z)\Big)=   -\sum_{j=1}^{n-k} \frac {\partial_z^{k+1}L_n}{\partial_x\partial_z^{k}L_n}(x_j,z) \frac{1}{( x-x_j)^2} =   \sum_{j=1}^{n-k}  \frac{dx_j}{dz} \frac{1}{( x-x_j)^2} >0.
\end{align}
 using that all $x_j$ are increasing by induction; and evaluating at $x=\tilde{x_i}(z)$ we get:
\[
 \partial_x\Big(\frac{\partial_z^{k+1}L_n}{\partial_z^{k}L_n}(\tilde{x_i}(z),z)\Big)>0.
\]

On the other hand, we know by Theorem~\ref{Laguerrez} that, for $x\in [0,+\infty[$, $\partial_z^{k}L_n(x,z)$ is realrooted with simple roots in $z$ as the $k^{th}$ derivative of a realrooted polynomial with simple roots. Therefore we can apply Laguerre's inequality (\ref{Laguerrein}) and get  $ [\partial_{z}^{k+2}L_n\partial_{z}^{k}L_n-(\partial_{z}^{k+1}L_n)^2 ](\tilde{x_i},z)< 0$ for all $z$, when evaluating at $x=\tilde{x_i}(z)$, which belongs to $[0,+\infty[$ by interlacing of the roots (the roots of the successive polynomials $\partial_{z}^{k}L_n$ are inside the convex hull of roots of $L_n$ ). It follows that 
\[
\partial_z \Big( \frac{\partial_{z}^{k+1}L_n}{\partial_{z}^{k}L_n}\Big)(\tilde{x_i}(z),z)  < 0. 
\]
We conclude by gathering the two inequalities. \\
\end{proof}
The induction is complete and Theorem~\ref{LaguerreD} is proven.

\subsection{Gegenbauer case}
In this subsection we prove Theorem~\ref{GegenbauerD}. Consider $z>-1/2$.
We have
\
\begin{align}
\hat{G}_n(-x,z)&=(-1)^n\hat{G}_n(x,z)  &    \partial_z\hat{G}_n(-x,z)&=(-1)^n\partial_z\hat{G}_n(x,z) .
\end{align}
So that if we there is a root $\tilde{y}(z)$ of $\partial_z\hat{G}_n(x,z)$ in $x$, then its symmetric counterpart $-\tilde{y}(z)$ will also be a root of  $\partial_z\hat{G}_n(x,z)$. In other terms, $\partial_z\hat{G}_n(x,z) $ and more generally , $\partial_z^k\hat{G}_n(x,z) $ have symmetric roots in $x$.

We show inductively that for $0\leq k \leq n$, $\partial_{z}^{k}\hat{G}_n(x,z)$  is realrooted of degree $n$ with $\lfloor n/2 \rfloor$ positive roots for $k \leq n- \lfloor n/2 \rfloor$, and for larger $k$, $n-k $ positive roots (for $k \geq  n- \lfloor n/2 \rfloor$, zero roots add two by two when $k$ increases by one). Also we prove that these positive roots are decreasing with $z$, and that for $k \geq 1$, the positive roots of  $\partial_{z}^{k}\hat{G}_n(x,z)$  interlace strictly the roots of $\partial_{z}^{k-1}\hat{G}_n(x,z)$ by below (that is the largest root in module belongs to $ \partial_{z}^{k-1}\hat{G}_n(x,z)$).

The initial case $k=0$ is already well known, as $\hat{G}_n(x,z)$ is of degre $n$ and has $\lfloor n/2 \rfloor$ positive roots that are decreasing with $z$ using Theorem~\ref{monoroots}. Let us assume it is true up to $0\leq k<n$. 

\begin{lemma} 
The  roots of  $\partial_z^{k+1}\hat{G}_n(x,z)$ of degree $n$ in $x$ are real  and the positive ones interlace  strictly those of $\partial_z^{k}\hat{G}_n(x,z)$ by below (that is the largest root in module belongs to $ \partial_z^{k}\hat{G}_n(x,z)$). There are $\lfloor n/2 \rfloor$ positive roots for $k+1 \leq n- \lfloor n/2 \rfloor$. For $k+1> n- \lfloor n/2 \rfloor$, we have $n-(k+1)$ positive roots. 
\end{lemma}

\begin{proof} 
We have
\begin{align}\label{gegsignb}
\partial_{z}^{k+1}\hat{G}_n(x,z)=\sum_{j=0}^{\lfloor n/2 \rfloor} (-1)^j \frac{ \partial_z^{k+1}\big[\prod_{i=n-\lfloor n/2 \rfloor}^{n-j-1}(z+i)\prod_{i=\lfloor n/2 \rfloor}^{n-1}(z+1/2+i)\big]}
 {j! (n-2j)!} x^{n-2j}2^{-2j}.
\end{align}
From this expression it follows that $\partial_z^{k+1}\hat{G}_n(x,z)$ is of degree $n$ in $x$. The coefficient of $x^n$ is indeed 
$ \frac{1}{n!} \partial_z^{k+1}\big[\prod_{i=n-\lfloor n/2 \rfloor}^{n-1}(z+i)\prod_{i=\lfloor n/2 \rfloor}^{n-1}(z+1/2+i)  \big]>0$ because $k<n$ and $z>-1/2$. 
\\

Assume first that $k+1\leq n- \lfloor n/2 \rfloor$. 
Denote the positive roots of $\partial_z^{k}\hat{G}_n(x,z)$ by $y_i$, in decreasing order. Then by differentiating the equality $\partial_z^{k}\hat{G}_n(y_i(z),z)=0$ with respect to $z$ we get 
\begin{align}\label{gegsignbis} 
\partial_z^{k+1}\hat{G}_n(y_i(z),z)= -\partial_x \big[\partial_z^{k}\hat{G}_n(y_i(z),z)\big] \frac{d}{dz}y_i.
\end{align}
 Let us distinguish between the even and odd cases.\\
 In the even case, we have (by induction) $n/2$ positive roots for $\partial_z^{k}\hat{G}_n(x,z)$,  $y_1(z)>y_2(z)\dots>y_{n/2}(z)>0$. On the one hand, using induction we have $\frac{d}{dz}y_i(z)< 0$. On the other hand, as the roots $y_i(z)$ are distinct, the derivative $\partial_x\big[\partial_z^{k}\hat{G}_n(y_i(z),z)\big] \neq 0$ changes sign when we increment $i$. We get that  $\partial_z^{k+1}\hat{G}_n(y_i(z),z)$ changes sign $n/2-1$ times using Equation~\ref{gegsignbis} and therefore using the intermediate value theorem  we get $n/2-1$ real zeros $ \tilde{y}_i(z)$ for $\partial_z^{k+1}\hat{G}_n(x,z)$ strictly between zeros of $\partial_z^{k}\hat{G}_n(x,z)$:  $y_1(z)>\tilde{y_1}(z)> y_2(z)>\tilde{y_2}(z)\dots>\tilde{y}_{n/2-1}(z)> y_{n/2}(z)>0$.
 Now there is still one root missing on the positive side. the sign of $\partial_x\partial_z^{k}\hat{G}_n(y_{1},z)$ is $+1$ (at the largest root the derivative is positive), and  $sign \big[\partial_x\partial_z^{k}\hat{G}_n(y_{n/2}(z),z)\big]= (-1)^{n/2-1}$ by the alternation of signs of the derivative from root to root. From there we get that $sign\big[\partial_z^{k+1}\hat{G}_n(y_{n/2}(z),z)\big]= (-1)^{n/2-1}$ using Equation~\ref{gegsignbis}. But we also have $\partial_z^{k+1}\hat{G}_n(0,z)=(-1)^{n/2}\frac{2^{-n}}{n/2!} \partial_z^{k+1} \big[\prod_{i= n/2}^{n-1}(z+1/2+i)\big]$ and for $z>-1/2$, $\partial_z^{k+1} \big[\prod_{i=n/2}^{n-1}(z+1/2+i)\big]>0$ for $z>-1/2$  because the product is of degree $n/2$ and $k+1 \leq n/2$ by assumption. Consequently,  $sign\big[\partial_z^{k+1}\hat{G}_n(0,z)\big]= (-1)^{n/2}$. 
 As  $\partial_z^{k+1}\hat{G}_n(y_{n/2}(z),z)$ and  $\partial_z^{k+1}\hat{G}_n(0,z)$ have opposite signs, there must exist an additional zero $\tilde{y}_{n/2}$ of $\partial_z^{k+1}\hat{G}_n$ in $]0, y_{n/2}(z)[ $. As a consequence we have $n/2$ positive roots. As $\partial_z^{k+1}\hat{G}_n$ is symmetric in $x$, we get $n/2$ additional zeros $-\tilde{y}_{i}(z)$. We get $n$ roots overall. \\
 
  In the odd case $n=2n_0+1$, denote the positive roots of  $\partial_z^{k}\hat{G}_n$ by $y_1(z)>y_2(z)\dots>y_{n_0}(z)>0$. If we collect as above the roots we get by the change of signs in Equation~\ref{gegsignb}, we get $n_0-1$ positive roots for $\partial_z^{k+1}\hat{G}_n$ interlacing with roots of $\partial_z^{k}\hat{G}_n$. We can't use directly the sign change between $0$ and $y_{n_0}$ as $0$ is a root, but we can look for the sign in the neighborhood of the origin. Indeed, we have using Equation~\ref{gegsignb}, $\partial_z^{k+1}\hat{G}_n(x,z) =_{x=0} x (-1)^{n_0} \frac{2^{-2n_0}}{n_0!} \partial_z^{k+1} \big[\prod_{i=n_0}^{2n_0}(z+1/2+i)\big]  + o(x)$. As $k+1\leq n_0+1$ by assumption and $z>-1/2$, $\partial_z^{k+1} \big[\prod_{i=n_0}^{2n_0}(z+1/2+i)\big]>0$. So that for  $x>0$ close to $0$, $sign\big[\partial_z^{k+1}\hat{G}_n(x,z)\big]=(-1)^{n_0} $, and on the other hand  $ sign\big[\partial_z^{k+1}\hat{G}_n(y_{n_0},z)\big]= sign\big[\partial_x\partial_z^{k}\hat{G}_n(y_{n_0},z)\big]= (-1)^{n_0-1} $ by the same sign alternation reasoning as in the even case. As a consequence we get an additional zero for $\partial_z^{k+1}\hat{G}_n$ in $]0,y_{n_0}[ $. Finally we get $n_0=\lfloor n/2 \rfloor$ positive roots. By symmetry, we also have $n_0$ negative roots, and also a root at $0$ which gives $n$ roots. Consequently, all the zeros of $\partial_z^{k+1}\hat{G}_n(x,z)$ have been found and the positive ones are interlacing  the positive zeros of $\partial_z^{k}\hat{G}_n(x,z)$ (strictly).\\

 Consider now the case $k \geq n- \lfloor n/2 \rfloor$. A new phenomenon arises, as the lowest degree terms in $x$ (which are also of lowest degree in $z$) vanish when we apply $\partial_{z}^{k}$ or  $\partial_{z}^{k+1}$. The factor with the lowest degree in $z$ is $\prod_{i=\lfloor n/2 \rfloor}^{n-1}(z+1/2+i)$ (coefficient of $x^{n- 2\lfloor n/2 \rfloor}$), which cancels for $k +1 > n- \lfloor n/2 \rfloor$. More precisely, for  $l= k-(n- \lfloor n/2 \rfloor)$, the $l$ first terms in the expansion of $\hat{G}_n(x,z)$ in $x$ vanish when we apply $\partial_{z}^{k}$. As a result, we have:

   \[
\partial_{z}^{k}\hat{G}_n(x,z)= \sum_{j=0}^{\lfloor n/2 \rfloor-l} (-1)^j \frac{ \partial_z^{k} \big[\prod_{i=n-\lfloor n/2 \rfloor}^{n-j-1}(z+i)\prod_{i=\lfloor n/2 \rfloor}^{n-1}(z+1/2+i)\big]}{j! (n-2j)!} (2x)^{n-2j}
  \]
  and 
   \[
\partial_{z}^{k+1}\hat{G}_n(x,z)= \sum_{j=0}^{\lfloor n/2 \rfloor-l-1} (-1)^j \frac{ \partial_z^{k+1} \big[\prod_{i=n-\lfloor n/2 \rfloor}^{n-j-1}(z+i)\prod_{i=\lfloor n/2 \rfloor}^{n-1}(z+1/2+i)\big]}{j! (n-2j)!} (2x)^{n-2j}.
  \]
  In the case $n$ is even, then the number of (distinct) positive  zeros for $\partial_{z}^{k}\hat{G}_n(x,z)$ is by induction $n-k=n/2-l$, the number of negative zeros is the same by symmetry and  the multiplicity of $0$ is $2l$ (which can be seen on the polynomial expansion). Use of the sign rule (through Equation~\ref{gegsignbis}) leads to  $n/2-l-1$ positive zeros and   $n/2-l-1$ negative zeros for  $\partial_z^{k+1}\hat{G}_n$. But  as the multiplicity of $0$ in  $\partial_z^{k+1}\hat{G}_n$   is $2(l+1)$ (we cancel a term in the summation when differentiating with respect to z) , then we get overall $n$ zeros, and the positive ones interlace strictly those of $\partial_z^{k}\hat{G}_n(x,z)$. \\
   In the odd case, $n=2n_0+1$,  the number of (distinct) positive zeros for $\partial_{z}^{k}\hat{G}_n(x,z)$ is by induction  $n-k=n_0-l$, the number of negative zeros is also $n_0-l$ ( and the multiplicity of $0$ is $2l+1$). Again, use of the sign rule (through Equation~\ref{gegsignbis}) leads to $n_0-l-1$ positive zeros and $n_0-l-1$ negative zeros  for  $\partial_z^{k+1}\hat{G}_n$. As the multiplicity of $0$ is $2(l+1)+1$, we also get overall $n$ zeros whose positive ones interlace strictly those of $\partial_z^{k}\hat{G}_n(x,z)$ . 
   
  \end{proof}

We then need to prove that the positive roots $\tilde{y}_i(z)$ of $\partial_z^{k+1}\hat{G}_n(x,z)$  share the monotonicity property.  

\begin{lemma}
For $1 \leq i \leq \lfloor n/2 \rfloor $ if $k < n- \lfloor n/2 \rfloor$ and $1 \leq i \leq n-(k+1)$ if $k \geq n- \lfloor n/2 \rfloor$, we have
$\frac{d \tilde{y_i}}{dz} < 0$.

\end{lemma}

\begin{proof}
We can assume that $k<n-1$, for if $k=n-1$, there is no positive root $\tilde{y_i}$. \\
First consider the case where all the roots of  are distinct, that is  $k < n- \lfloor n/2 \rfloor$. 
By assumption, as $\partial_z^{k}\hat{G}_n(\tilde{y_i}(z),z) \neq 0$ (strict interlacing for positive roots),
\[
   \frac{\partial_z^{k+1}\hat{G}_n(\tilde{y_i}(z),z)}{\partial_z^{k}\hat{G}_n(\tilde{y_i}(z),z)}=0
\]
which leads to
\begin{align}
\frac{d \tilde{y_i}}{dz} =-\frac{ \partial_z \Big( \frac{\partial_{z}^{k+1}\hat{G}_n}{\partial_z^{k}\hat{G}_n}(\tilde{y_i}(z),z)\Big)}{\partial_x \Big(\frac{\partial_{z}^{k+1}\hat{G}_n}{\partial_z^{k}\hat{G}_n}(\tilde{y_i}(z),z)\Big)}.
\end{align}
If we denote by $\delta_k \geq 1$ the number of positive roots of $\partial_{z}^{k}\hat{G}_n(x,z)$, which is equal to $\lfloor n/2 \rfloor$ or $n-k$ according to $k$, and if $y_j$, $j\leq \delta_k$ represent the distinct positive roots of $\partial_{z}^{k}\hat{G}_n(x,z)$, we have by fraction decomposition (as the degrees of $\partial_{z}^{k+1}\hat{G}_n$ and $\partial_{z}^{k}\hat{G}_n$ are both equal to $n$):
\[
\frac{\partial_{z}^{k+1}\hat{G}_n}{\partial_{z}^{k}\hat{G}_n}(x,z)=  c(z)+  \sum_{j=1}^{\delta_k}\frac {\partial_{z}^{k+1}\hat{G}_n}{\partial_x\partial_{z}^{k}\hat{G}_n}(y_j,z) \frac{1}{x-y_j} +   \sum_{j=1}^{\delta_k}\frac {\partial_{z}^{k+1}\hat{G}_n}{\partial_x\partial_{z}^{k}\hat{G}_n}(-y_j,z) \frac{1}{x+y_j},
\]
where $c(z)$ is the quotient of the two coefficients of $x^n$ in our polynomials. 
The zero roots (with possibly multiplicity greater than 1) indeed cancel out in the fraction because $\partial_{z}^{k+1}\hat{G}_n(x,z)$ has $0$ as a root with multiplicity higher than $\partial_{z}^{k}\hat{G}_n(x,z)$, so that we don't have to consider $0$ as a pole in the fraction decomposition. 
We can differentiate this equality in $x$:
\[
\partial_x\Big(\frac{\partial_{z}^{k+1}\hat{G}_n}{\partial_{z}^{k}\hat{G}_n}(x,z)\Big)=  - \sum_{j=1}^{\delta_k}\frac {\partial_{z}^{k+1}\hat{G}_n}{\partial_x\partial_{z}^{k}\hat{G}_n}(y_j,z) \frac{1}{(x-y_j)^2} -   \sum_{j=1}^{\delta_k}\frac {\partial_{z}^{k+1}\hat{G}_n}{\partial_x\partial_{z}^{k}\hat{G}_n}(-y_j,z) \frac{1}{(x+y_j)^2},
\]
and noticing that $-\frac {\partial_{z}^{k+1}\hat{G}_n}{\partial_x\partial_{z}^{k}\hat{G}_n}(y_j,z)= \frac{d y_j(z)}{dz}$, and  $-\frac {\partial_{z}^{k+1}\hat{G}_n}{\partial_x\partial_{z}^{k}\hat{G}_n}(-y_j,z)= -\frac{d y_j(z)}{dz}$, we get, evaluating at $x=\tilde{y_i}(z)$,
\begin{align}
\partial_x\Big(\frac{\partial_{z}^{k+1}\hat{G}_n}{\partial_{z}^{k}\hat{G}_n}(\tilde{y_i}(z),z)\Big)&= \sum_{j=1}^{\delta_k}\frac{dy_j}{dz} \big[ \frac{1}{(\tilde{y_i}-y_j)^2}- \frac{1}{(\tilde{y_i}+y_j)^2} \big] \\
&= 4\tilde{y_i} \sum_{j=1}^{\delta_k}\frac{\frac{dy_j}{dz} y_j}{(\tilde{y_i}-y_j)^2(\tilde{y_i}+y_j)^2}.
\end{align}
 As we have $\tilde{y_i}>0$, and for all $j$ such that $y_j > 0$, $\frac{d y_j}{dz}y_j<0$,
 \[
\partial_x\Big(\frac{\partial_{z}^{k+1}\hat{G}_n}{\partial_{z}^{k}\hat{G}_n}(\tilde{y_i}(z),z)\Big)< 0.
\]

On the other hand, for all $y_0 \in [-1,1]$,  ${\hat{G}_n}(y_0,z)$ is realrooted in $z$ with simple roots (according to Corollary~\ref{Gegenbauerzmod}), therefore $\partial_{z}^{k}\hat{G}_n(y_0,z)$ is also realrooted with simple roots as  the $k^{th}$ derivative in $z$, and we get that $ [\partial_{z}^{k+2}\hat{G}_n\partial_{z}^{k}\hat{G}_n-(\partial_{z}^{k+1}\hat{G}_n)^2](y_0,z) < 0$ for all $z$ using Lemma~\ref{Laguerrein}. Specializing for $y_0=\tilde{y_i}(z)$ which by interlacing does belong to  $[-1,1]$, we have for all $z$,    $ [\partial_{z}^{k+2}\hat{G}_n\partial_{z}^{k}\hat{G}_n-(\partial_{z}^{k+1}\hat{G}_n)^2](\tilde{y_i},z)< 0$ from where it follows that
\[
\partial_z\Big(\frac{\partial_{z}^{k+1}\hat{G}_n}{\partial_{z}^{k}\hat{G}_n}(\tilde{y_i}(z),z)\Big) < 0.
\]

We conclude by gathering the two inequalities.
\end{proof}
Eventually, we proved Theorem~\ref{GegenbauerD} by induction.


\begin{thebibliography}{9}
\bibitem{Alfa}
M. Alfaro, M.A. de Morales, M.L. Rezola ,\textit{Orthogonality of the Jacobi polynomials with negative integer parameters}, Elsevier, Journal of Computational and Applied Mathematics(2002). 
\bibitem{DOP}
A. Nikiforov, S. Suslov, V. Uvarov, \textit{Classical Orthogonal Polynomials of a Discrete Variable}
\bibitem{G}
A. Gribinski, \textit{Rectangular finite free probability theory}, Journal of Theoretical Probability, Springer.
\bibitem{GM}
A. Gribinski, A. Marcus, \textit{A rectangular additive convolution for polynomials}, Journal of Combinatorial Theory (2022).
\bibitem{KUI}
A.Kuijlaars, KTR. McLaughlin,\textit{Asymptotic Zero Behavior of Laguerre Polynomials with Negative Parameter}, Constructive Approximation 20 (2004).
\bibitem{KVA}
A.Kuijlaars, W. Van Assche,\textit{The Asymptotic Zero Distribution of Orthogonal Polynomials with Varying Recurrence Coefficients}, Journal of Approximation Theory (1999) .
\bibitem{MSS}
A. Marcus, D. Spielman, N. Srivastava, \textit{Finite free convolutions of polynomials}, Probab. Theory Relat. Fields 182, 807-848 (2022). 
\bibitem{Szego}
G. Szego, \textit{Orthogonal polynomials}, AMS, Providence, RI MR 51 (1975): 8724.
\end{thebibliography}
\end{document}